\documentclass{article}
\usepackage{amssymb}
\usepackage{amsfonts}
\usepackage{amsmath}

\usepackage{tikz}

\newtheorem{theorem}{Theorem}

\newtheorem{lemma}{Lemma}

\newtheorem{proposition}{Proposition}

\newenvironment{proof}[1][Proof]{\noindent\textbf{#1.} }{\ \rule{0.5em}{0.5em}}
\input{tcilatex}
\begin{document}

\title{Vector variational problem with knitting boundary conditions}
\author{Gra\c{c}a Carita, Vladimir V. Goncharov\thanks{%
CIMA, Universidade de \'{E}vora, Rua Rom\~{a}o Ramalho 59, 7000-671, \'{E}%
vora, Portugal; V. Goncharov also belongs to the Institute of System
Dynamics and Control Theory, RAS, ul. Lermontov 134, 664033, Irkutsk,
Russia; e-mails: gcarita@uevora.pt (G. Carita), goncha@uevora.pt (V.
Goncharov)} \and and Georgi V. Smirnov\thanks{%
Universidade do Minho; Braga, Portugal; e-mail: smirnov@uminho.pt}}
\date{}
\maketitle

\begin{abstract}
We consider a variational problem with a polyconvex integrand and
nonstandard boundary conditions that can be treated as minimization of the strain
energy during the suturing  process in the plastic surgery. Existence of
minimizers is proved as well as  necessary optimality conditions are 
discussed.


\end{abstract}

\thispagestyle{empty}

\bigskip

\section{Introduction}

\bigskip

Given an open bounded connected domain $\Omega \subset \mathbb{R}^{N}$ with
a sufficiently regular (locally Lipschitz) boundary, $\partial \Omega $, let
us consider the integral%
\begin{equation}
I\left( u\right) :=\int_{\Omega }W\left( \nabla u\left( x\right) \right) \,dx
\label{Functional}
\end{equation}%
to be minimized on a class of Sobolev functions $u:\Omega \rightarrow 
\mathbb{R}^{d}$ with a kind of boundary conditions to be described later.
All over the paper we assume the integrand $W:\mathbb{R}^{d\times
N}\rightarrow \mathbb{R\cup }\left\{ +\infty \right\} $ to be \emph{%
polyconvex}. This means that the representation%
\begin{equation*}
W\left( \xi \right) =g\left( \mathbb{T}\left( \xi \right) \right) 
,\;\; \xi \in \mathbb{R}^{d\times N},
\end{equation*}
holds for some convex function $g:\mathbb{R}^{\tau \left( d,N\right)
}\rightarrow \mathbb{R\cup }\left\{ +\infty \right\} $,%
\begin{equation*}
\tau( d,N) :=\underset{s=1}{\overset{d\wedge N}{\dsum }}
\varkappa(s) ,\;\; \varkappa(s) :=\dbinom{s%
}{d}\dbinom{s}{N}=\frac{d!N!}{(s!) ^{2}( d-s)!(
N-s) !},
\end{equation*}
where%
\begin{equation*}
\mathbb{T}\left(\xi \right) :=\left( \mathrm{Adj}_{1}\xi ,\mathrm{Adj}%
_{2}\xi ,\mathrm{Adj}_{3}\xi ,\dots ,\mathrm{Adj}_{d\wedge N}\xi \right) 
,\;\; \xi \in \mathbb{R}^{d\times N},
\end{equation*}%
and $\mathrm{Adj}_{k}\xi $ is the vector of all \emph{minors} of the matrix $%
\xi $ of order $k=1,2,\dots ,d\wedge N$, respectively. In particular, $%
\mathrm{Adj}_{1}\xi =\xi $ and $\mathrm{Adj}_{d}\xi =\det \,\xi $ whenever $%
d=N$.

It is known that, under strong coercivity assumptions on $W$ to assure weak convergence of the minors of gradients for the minimizing sequence, the functional $I$ attains its minimum on $\bar{u}\left( \cdot \right) +\mathbf{W%
}_{0}^{1,p}\left( \Omega ;\mathbb{R}^{d}\right) $, $p\geq 1$. We refer to the fundamental work by J. Ball \cite{B} motivated by problems coming from nonlinear elasticity and to \cite{ADM, MQY, My} for further improvements.

The lower semicontinuity for general polyconvex integrands with respect to the weak convergence in $W^{1,p}(\Omega;\mathbb{R}^{d}),$       $\Omega\subset \mathbb{R}^{N},$ has been the subject of many investigations.
Namely, Marcellini showed in \cite{M} that this property holds whenever $p<N$. Later, this result was improved by Dacorogna and Marcellini in \cite{DM} who proved the lower semicontinuity for $p>N-1$ while Mal\'{y} in \cite{My} exhibited a counterexample for $p<N-1$. The limit case $p=N-1$ was adressed in \cite{ADM, DMS, CDM, FH}. Very recently (see \cite{FFLMMV} and \cite{DPMF})  the limit case was studied for polyconvex integrands depending on $x$ and/or on $u$.

Besides the Dirichlet boundary condition $u=%
\bar{u}$ for the displacement one considered, for instance, boundary condition on traction, which somehow depends on the normal
derivatives of $u\left( \cdot \right) $ (generalizing the Neumann boundary
data). Observe that these conditions (displacement, traction or a mixed one)
can be applied either to the whole boundary $\partial \Omega $, or to some
of its subsets of positive Hausdorff measure leaving the rest free.
Moreover, some restrictions on the \emph{jacobian} may be relevant and
practically justified. For example, the constraint $\det \,\nabla u\left(
x\right) >0$ means that the minimum is searched among the deformations
preserving orientation, while $\det \,\nabla u\left( x\right) =1$ refers to
the case of incompressible elastic body.

One of the possible applications of the above variational problem is
regarded to plastic surgery, namely, in the woman breast reduction, where we
deal with a sort of very elastic and soft tissue. Some recent publications
(see, e.g., \cite{Ay, D, S1, PY, S2}) were devoted to mathematical setting
of the related problems and to their numerical simulations. Medical
examinations allow to consider the involved tissue as a neo-Hookean
compressible material (see \cite{Z}). We have a more precise model when the strain energy is  defined
by the integral (\ref{Functional}) with the density $W:\mathbb{R}^{3\times
3}\rightarrow \mathbb{R}$,%
\begin{eqnarray}
&&W\left( \xi \right) :=\mu \left( \limfunc{tr}\left( \xi \cdot \xi
^{T}\right) -3-2\ln \left( \det \xi \right) \right)   \notag \\
&&+\lambda \left( \det \xi -1\right) ^{2}+\beta \limfunc{tr}\mathrm{\,}%
\left( \mathrm{Adj\,}\xi \cdot \mathrm{Adj\,}\xi ^{T}\right) \,,
\label{Int_ex}
\end{eqnarray}%
where "$\limfunc{tr}$" means the trace of a matrix, $\mathrm{Adj\,}\xi :=%
\mathrm{Adj}_{2}\xi $, and the symbol "$T$" stands for the matrix
transposition.  One of the steps of the
(breast reduction) surgery is the suturing, which mathematically can
be seen as an identification of points of some surface piece $\Gamma
^{+}\subset \partial \Omega $ with points of another one $\Gamma ^{-}\subset
\partial \Omega $. Denoting the respective correspondence between the points
of $\Gamma ^{+}$ and $\Gamma ^{-}$ by $\sigma $, we are led to a new type of constraint%
\begin{equation}
u(x) =(u\circ \sigma)(x), \;x\in
\Gamma ^{+},  \label{sutur_cond}
\end{equation}%
called the \emph{knitting boundary condition}. Let us note that the one-to-one
mapping $\sigma $ is not \emph{a priori} given and should be chosen to
guarantee the minimum value to the functional (\ref{Functional}). In other
words, a minimizer of (\ref{Functional}) (if any) should be a pair $\left(
u,\sigma \right) $ where $u\in \mathbf{W}^{1,p}\left( \Omega ;\mathbb{R}%
^{3}\right) $, $p\geq 1$, and $\sigma :\Gamma ^{+}\rightarrow \Gamma ^{-}$
is sufficiently regular. We set the natural hypothesis that $\sigma $ and
its inverse $\sigma ^{-1}$ are Lipschitz transformations (with the same
Lipschitz constant $L>0$). Practically this means that the sutured tissue
can not be extended nor compressed too much.

Motivated by the problem coming from the plastic surgery we will consider
just the case $p=2$ and $d=N=3$, although the results remain
true in the case $p>2$ and arbitrary $d=N\geq 2$ as well.

The paper is organized as follows. In the next section we give the exact
setting of the variational problem together with the main hypotheses on the
integrand $W$. For simplicity of references we put here also some important
facts regarded with the Sobolev functions. In Section 3 we justify first the
well-posedness of the problem by showing that the composed function from the
knitting condition (\ref{sutur_cond}) belongs to the respective Lebesgue
class. Afterwards, we prove existence of a minimizer as an accumulation
point of an arbitrary minimizing sequence (the so called \emph{direct method}%
, see \cite{Dac}). The paper is concluded with a necessary optimality
condition for the given problem (see Section 4) allowing to construct
effective numerical algorithms, which can be successively applied in the
medical practice.

\bigskip

\section{\protect\bigskip Main hypotheses and auxiliary results}

In what follows we fix a nonempty open bounded and connected set $\Omega
\subset \mathbb{R}^{3}$ whose boundary $\partial \Omega $ is assumed to be
locally Lipschitz (see, e.g., \cite[p. 354]{L}). By the symbol $\mathcal{L}%
^{m}$ ($dx$) we denote the Lebesgue measure in the space $\mathbb{R}^{m}$, $%
m=2,3$, while $\mathcal{H}^{2}$ means the two-dimensional \emph{Hausdorff
measure} (see, e.g., \cite{F}).

Let us divide the surface $\partial \Omega $ into several parts $\Gamma _{i}$%
, $i=1,2,3,4$, in such a way that $\mathcal{H}^{2}\left( \Gamma _{i}\cap
\Gamma _{j}\right) =0$ for $i\neq j$. Moreover, we set $\Gamma _{4}:=\Gamma
^{+}\cup \Gamma ^{-}$ where $\Gamma ^{\pm }\subset \Gamma $ with $\mathcal{H}%
^{2}\left( \Gamma ^{\pm }\right) >0$ and $\mathcal{H}^{2}\left( \Gamma
^{+}\cap \Gamma ^{-}\right) =0$ are also given.

Suppose that $W:\mathbb{R}^{3\times 3}\rightarrow \mathbb{R}$ is a \emph{%
polyconvex function} satisfying the \emph{growth\ assumption}:%
\begin{equation}
W(\xi) \geq c_{0}+c_{1}\left\vert \xi \right\vert
^{2}+c_{2}\left\vert \mathrm{Adj\,}\xi \right\vert ^{2}+c_{3}\left(\det \xi
\right)^{2}, \;\; \xi \in \mathbb{R}^{3\times 3},
\label{growth_cond}
\end{equation}%
where $c_{0}\in \mathbb{R}$ and $c_{i}>0$, $i=1,2,3$, are some given
constants. Here and in what follows by $\left\vert \cdot \right\vert $ we
denote the norm of both a vector in $\mathbb{R}^{n}$ and a $3\times 3$%
-matrix.

Taking into account that $\limfunc{tr}\left( \xi \cdot \xi ^{T}\right)
=\left\vert \xi \right\vert ^{2}$ for each matrix $\xi \in \mathbb{R}%
^{3\times 3}$, we see that the integrand (\ref{Int_ex}) satisfies the above
properties. Indeed, it is convex as a function of $\mathbb{T}\left( \xi
\right) $ being represented as a sum of three terms, which are convex w.r.t. 
$\xi $, $\det \xi $ and $\mathrm{Adj\,}\xi $, respectively. Furthermore,%
\begin{equation*}
W(\xi) =-3\mu +\lambda +\mu \left\vert \xi \right\vert
^{2}+\beta \left\vert \mathrm{Adj\,}\xi \right\vert ^{2}+f( \det \xi), \;\xi \in \mathbb{R}^{3\times 3},
\end{equation*}%
where the function%
\begin{equation*}
f\left( t\right) :=\frac{\lambda }{2}t^{2}-2\lambda t-2\mu \ln t, \;t>0%
,
\end{equation*}%
is lower bounded by some (negative) constant.

Since on various pieces of the surface $\partial \Omega $ the boundary conditions
are structurally different (some part of $\partial \Omega $ can be left even
free), to set the problem we use the notion of the \emph{trace operator},
which associates to each $u\in $ $\mathbf{W}^{1,2}\left( \Omega ;\mathbb{R}%
^{3}\right) $ a function $\limfunc{Tr}u$ defined on the boundary, $\partial
\Omega,$ which can be interpreted as the "\emph{boundary values}" of $u$. We refer to \cite[pp. 465-474]{L}, where
the existence and uniqueness of the trace operator were proved for scalar Sobolev functions $u \in \mathbf{W}^{1,p}(\Omega),   p>1.$ For vector-valued functions $u: \Omega\rightarrow \mathbb{R}^{3}$ instead, we can argue componentwise. So, applying \cite[Theorem 15.23]{L}, we define the trace as the linear and bounded operator
$\limfunc{Tr}:\mathbf{W}^{1,2}\left( \Omega ;\mathbb{R}%
^{3}\right) \rightarrow \mathbf{L}^{2}(\partial \Omega ;\mathbb{R}%
^{3})$ satisfying the following properties:

\begin{enumerate}
\item $\limfunc{Tr}u\left( x\right) =u\left( x\right) $, $x\in \partial
\Omega $, whenever $u\in \mathbf{W}^{1,2}\left( \Omega ;\mathbb{R}%
^{3}\right) \cap \mathbf{C}\left( \overline{\Omega };\mathbb{R}^{3}\right) $;

\item for each $u\in \mathbf{W}^{1,2}(\Omega;\mathbb{R}^3)$
and any test function $\varphi \in \mathbf{C}^{1}(\overline{\Omega };\mathbb{R}^{3})$ the equalities%
\begin{equation*}
\int_{\Omega}u_{j}\frac{\partial \varphi_{j}}{\partial x_{i}}dx=-\int_{\Omega}\varphi_{j}\frac{\partial u_{j}}{\partial x_{i}}dx+\int_{\partial \Omega }\varphi_{j}
\limfunc{Tr}(u_j)\nu_i d\mathcal{H}^{2},%
\end{equation*}%
hold for each $ i, j=1,2,3$ where $\nu:=(\nu_1,\nu_2,\nu_3)^{T}$ means the unit outward normal to $%
\partial \Omega $.
\end{enumerate}
 
In addition to the properties above, observe that the trace operator gives a compact embedding into the space $\mathbf{L}^{2}\left( \partial \Omega ;\mathbb{R}^{3}\right)$ that will be crucial to obtain the main result in Section 3. Namely, the following proposition takes place.

\begin{proposition}
\label{traceconv}Let $\Omega \subset \mathbb{R}^{3}$ be an open bounded set
with locally Lipschitz boundary. Then for each $\left\{ u_{n}\right\}
\subset \mathbf{W}^{1,2}\left( \Omega ;\mathbb{R}^{3}\right) $ converging to 
$u$ weakly in $\mathbf{W}^{1,2}\left( \Omega ;\mathbb{R}^{3}\right) $ the
sequence of traces $\left\{ \limfunc{Tr}u_{n}\right\} $ converges to $%
\limfunc{Tr}u$ strongly in $\mathbf{L}^{2}\left( \partial \Omega ;\mathbb{R}%
^{3}\right)$.
\end{proposition}

The proof is based essentially on the following lemma giving a nice estimate
for the surface integral of the trace operator.

\begin{lemma}
\label{Lemmatrace}Let $\Omega \subset \mathbb{R}^{3}$ be as in Proposition \ref{traceconv}. Then there exists a constant
$C>0$ such that 
\begin{equation}
\int_{\partial \Omega }\left\vert \limfunc{Tr}u\right\vert ^{2}d\mathcal{H}%
^{2}\leq C\left( \frac{1}{\varepsilon }\int_{\Omega }\left\vert u\right\vert
^{2}dx+\varepsilon \int_{\Omega }\left\vert \nabla u\right\vert
^{2}dx\right)   \label{Estimate_tr}
\end{equation}%
for any $\varepsilon >0$ and any $u\in \mathbf{W}^{1,2}\left( \Omega ;\mathbb{R}^{3}\right) .$
\end{lemma}

\begin{proof}
Given $x\in \partial \Omega ,$ due to the Lipschitz hypothesis there exists
a neighborhood $U_{x}$ of $x$ such that $U_{x}\cap \partial \Omega $ can be
represented as the graph of a Lipschitz function w.r.t. some (local)
coordinates. Without loss of generality, we may assume that 
\begin{equation*}
U_{x}\cap \Omega =\left\{ \left( y^{\prime },y_{3}\right) :f_{x}\left(
y^{\prime }\right) <y_{3}\leq f_{x}\left( y^{\prime }\right) +\delta
_{x},~y^{\prime }\in G_{x}\right\} 
\end{equation*}%
and%
\begin{equation*}
U_{x}\cap \partial \Omega =\left\{ \left( y^{\prime },f_{x}\left( y^{\prime
}\right) \right) :y^{\prime }\in G_{x}\right\} ,
\end{equation*}%
where $\delta _{x}>0,$ $G_{x}\subset \mathbb{R}^{2}$ is an open set in the
space of the first two coordinates and $f_{x}:G_{x}\rightarrow \mathbb{R}$
is a Lipschitz function. By compactness there exists a finite number of
points $x^{i}\in \partial \Omega \,,~i=1,\dots ,q,$ such that 
\begin{equation*}
\partial \Omega =\tbigcup_{i=1}^{q}\left( U_{x^{i}}\cap \partial \Omega
\right) .
\end{equation*}%
Set $\delta _{i}:=\delta _{x^{i}},~U_{i}:=U_{x^{i}},~G_{i}:=G_{x^{i}}$ and~ $%
f_{i}:=f_{x^{i}}$, $i=1,\dots ,q$. Denote by $L>0$ the biggest Lipschitz
constant of the functions $f_{i}.$

Let us choose $\varepsilon <\min \left\{ \delta _{i}:i=1,\dots ,q\right\} $
and consider first the function $u\in \mathbf{C}^{1}\left( \overline{\Omega }%
\right) \cap \mathbf{W}^{1,2}\left( \Omega ;\mathbb{R}^{3}\right) .$ Given $%
i\in \left\{ 1,\dots ,q\right\} ,$ by the Newton-Leibniz formula, for each $%
x^{\prime }\in G_{i}$ and $\ x_{3}\in \mathbb{R}$ with $f_{i}\left(
x^{\prime }\right) \leq x_{3}<f_{i}\left( x^{\prime }\right) +\varepsilon $
we have 
\begin{equation*}
u\left( x^{\prime },f_{i}\left( x^{\prime }\right) \right) =u\left(
x^{\prime },x_{3}\right) -\int_{f_{i}\left( x^{\prime }\right) }^{x_{3}}%
\frac{\partial u}{\partial x_{3}}\left( x^{\prime },s\right) ~ds.
\end{equation*}%
In turn, by Cauchy-Schwartz inequality,%
\begin{eqnarray*}
\left\vert u\left( x^{\prime },f_{i}\left( x^{\prime }\right) \right)
\right\vert ^{2} &\leq &2\left( \left\vert u\left( x^{\prime },x_{3}\right)
\right\vert ^{2}+\left( \int_{f_{i}\left( x^{\prime }\right) }^{f_{i}\left(
x^{\prime }\right) +\varepsilon }\left\vert \frac{\partial u}{\partial x_{3}}%
\left( x^{\prime },s\right) \right\vert ds\right) ^{2}\right)  \\
&\leq &2\left( \left\vert u\left( x^{\prime },x_{3}\right) \right\vert
^{2}+\varepsilon \int_{f_{i}\left( x^{\prime }\right) }^{f_{i}\left(
x^{\prime }\right) +\varepsilon }\left\vert \frac{\partial u}{\partial x_{3}}%
\left( x^{\prime },s\right) \right\vert ^{2}ds\right) .
\end{eqnarray*}%
Hence,

\begin{eqnarray*}
&&\left\vert u\left( x^{\prime },f_{i}(x^{\prime })\right) \right\vert ^{2}\sqrt{%
1+\left\vert \nabla f_{i}\left( x^{\prime }\right) \right\vert ^{2}} \\
&\leq &2\sqrt{1+L^{2}}\left( \left\vert u\left( x^{\prime },x_{3}\right)
\right\vert ^{2}+\varepsilon \int_{f_{i}\left( x^{\prime }\right)
}^{f_{i}\left( x^{\prime }\right) +\varepsilon }\left\vert \frac{\partial u}{%
\partial x_{3}}\left( x^{\prime },s\right) \right\vert ^{2}ds\right) .
\end{eqnarray*}%
Integrating both parts of the previous inequality in the cylinder 
\begin{equation*}
\Omega _{i}:=\left\{ \left( x^{\prime },x_{3}\right) :f\left( x^{\prime
}\right) \leq x_{3}\leq f\left( x^{\prime }\right) +\varepsilon ,~x^{\prime
}\in G_{i}\right\} \subset U_{i}\cap \Omega ,
\end{equation*}%
gives%
\begin{eqnarray*}
\varepsilon \int_{U_{i}\cap \partial \Omega }\left\vert u\left( x\right)
\right\vert ^{2}d\mathcal{H}^{2}\left( x\right)  &\leq &2\sqrt{1+L^{2}}%
\left( \int_{\Omega _{i}}\left\vert u\left( x\right) \right\vert
^{2}dx+\varepsilon ^{2}\int_{\Omega _{i}}\left\vert \frac{\partial u}{%
\partial x_{3}}\left( x\right) \right\vert ^{2}dx\right)  \\
&\leq &2\sqrt{1+L^{2}}\left( \int_{\Omega }\left\vert u\left( x\right)
\right\vert ^{2}dx+\varepsilon ^{2}\int_{\Omega }\left\vert \frac{\partial u%
}{\partial x_{3}}\left( x\right) \right\vert ^{2}dx\right) .
\end{eqnarray*}

Summing in $i=1,\dots ,q$ we have%
\begin{eqnarray}
&&\int_{\partial \Omega }\left\vert u\left( x\right) \right\vert ^{2}d%
\mathcal{H}^{2}\left( x\right)   \label{traceu} \\
&\leq &2q\sqrt{1+L^{2}}\left( \frac{1}{\varepsilon }\int_{\Omega }\left\vert
u\left( x\right) \right\vert ^{2}dx+\varepsilon \int_{\Omega }\left\vert 
\frac{\partial u}{\partial x_{3}}\left( x\right) \right\vert ^{2}dx\right) .
\notag
\end{eqnarray}%
The inequality (\ref{Estimate_tr}) for each $u \in \mathbf{W}^{1,2}(\Omega ;\mathbb{R}^{3})$ follows now from (\ref{traceu}) by the
properties of the trace operator and by the density of smooth functions.
\end{proof}

\medskip

\begin{proof}[Proof of Proposition 1]
Given a sequence $\left\{ u_{n}\right\} \subset \mathbf{W}^{1,2}\left(
\Omega ;\mathbb{R}^{3}\right) $ converging weakly to $u$ in $\mathbf{W}%
^{1,2}\left( \Omega ;\mathbb{R}^{3}\right) $, there exists $M>0$ with 
\begin{equation*}
\int_{\Omega }\left\vert \nabla u_{n}\left( x\right) \right\vert ^{2}dx\leq M
\end{equation*}%
for all $n\in \mathbb{N}.$ Then, by the \emph{Rellich-Kondrachov theorem}
(see \cite[Theorem 11.21, p. 326]{L}), $u_{n}\rightarrow u$ $\ $strongly$\ $%
in$~\mathbf{L}^{2}\left( \Omega ;\mathbb{R}^{3}\right) .$ Applying Lemma \ref%
{Lemmatrace} to $\limfunc{Tr}\left( u_{n}-u\right) =\limfunc{Tr}u_{n}-%
\limfunc{Tr}u,$ we have%
\begin{align*}
\int_{\partial \Omega }\vert \limfunc{Tr}u_{n}-\limfunc{Tr}
u\vert ^{2}d\mathcal{H}^{2}&\leq C(\frac{1}{\varepsilon }\int_{\Omega }\vert
u_{n}-u\vert ^{2}dx+\varepsilon \int_{\Omega }\vert \nabla
u_{n}-\nabla u\vert ^{2}dx)\\
&\leq C( \frac{1}{\varepsilon }\int_{\Omega }\vert
u_{n}-u\vert ^{2}dx+4\varepsilon M) .
\end{align*}
Hence%
\begin{equation*}
\underset{n\rightarrow \infty }{\lim \sup }\int_{\partial \Omega }\left\vert 
\limfunc{Tr}u_{n}-\limfunc{Tr}u\right\vert ^{2}d\mathcal{H}^{2}\leq
4\varepsilon CM.
\end{equation*}%
Letting $\varepsilon \rightarrow 0^{+}$ concludes the proof.\medskip 
\end{proof}

We will use also the so called \emph{generalized
Poincar\'{e} inequality} (see \cite[Theorem 6.1-8, p. 281]{C}).

\begin{proposition}
\label{poincare}Given an open bounded domain $\Omega \subset \mathbb{R}^{3}$
with locally Lipschitz boundary and a measurable subset $\Gamma \subset
\partial \Omega $ with $\mathcal{H}^{2}\left( \Gamma \right) >0$ there
exists a constant $C>0$ such that 
\begin{equation}
\underset{\Omega }{\dint }\left\vert u\left( x\right) \right\vert ^{2}dx\leq
C\left[ \underset{\Omega }{\dint }\left\vert \nabla u\left( x\right)
\right\vert ^{2}\,dx+\left\vert \underset{\Gamma }{\dint }\limfunc{Tr}%
u\left( x\right) \,d\mathcal{H}^{2}\left( x\right) \right\vert ^{2}\right] 
\label{poincare_ineq}
\end{equation}%
for each $u\in \mathbf{W}^{1,2}\left( \Omega ;\mathbb{R}^{3}\right) $.
\end{proposition}

\medskip
Let us formulate now the boundary conditions in terms of the trace operator.
Consider first a surface $\mathcal{S}\subset \mathbb{R}^{3}$ defined by some
continuous function $h:\mathbb{R}^{3}\rightarrow \mathbb{R}$,%
\begin{equation*}
\mathcal{S}:=\left\{ u\in \mathbb{R}^{3}:h\left( u\right) =0\right\} .
\end{equation*}
Then, given $L\geq 1$ we denote by $\Sigma _{L}\left( \Gamma ^{+};\Gamma
^{-}\right) $ the set of all functions $\sigma :\Gamma ^{+}\rightarrow
\Gamma ^{-}$ satisfying the inequalities%
\begin{equation}
\frac{1}{L}\left\vert x-y\right\vert \leq \left\vert \sigma \left( x\right)
-\sigma \left( y\right) \right\vert \leq L\left\vert x-y\right\vert 
\label{Lipschitz}
\end{equation}%
for all $x,y\in \Gamma ^{+}$, and introduce the set $\mathcal{W}_{L}\subset 
\mathbf{W}^{1,2}\left( \Omega ;\mathbb{R}^{3}\right) \times \Sigma
_{L}\left( \Gamma ^{+};\Gamma ^{-}\right) $ of all pairs $\left( u,\sigma
\right) $ satisfying the (boundary) conditions:

\begin{enumerate}
\item[(C$_{1}$)] $\limfunc{Tr}u\left( x\right) =x$ for $\mathcal{H}^{2}$%
-a.e. $x\in \Gamma _{1}$;

\item[(C$_{2}$)] $h\left( \limfunc{Tr}u\left( x\right) \right) =0$ for $%
\mathcal{H}^{2}$-a.e. $x\in \Gamma _{2}$;

\item[(C$_{3}$)] $\limfunc{Tr}u\left( x\right) =\limfunc{Tr}u\left( \sigma
\left( x\right) \right) $ for $\mathcal{H}^{2}$-a.e. $x\in \Gamma ^{+}$.
\end{enumerate}
Thus, we can write the \emph{knitting variational problem} in the form%
\begin{equation}
\min \left\{ \int_{\Omega }W\left( \nabla u\left( x\right) \right)dx:\left( u,\sigma \right) \in \mathcal{W}_{L}\right\}.  \label{varproblem}
\end{equation}%

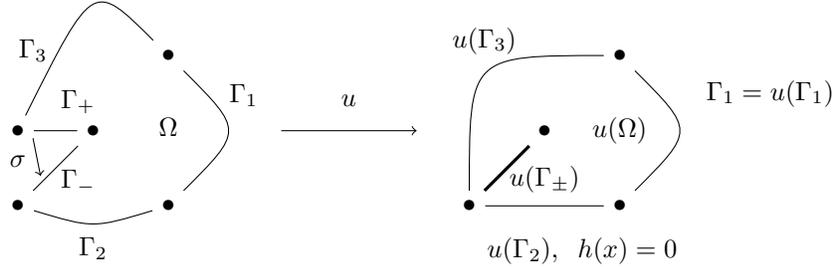
\begin{figure}[htb]	
\begin{center}
\begin{tikzpicture}
\node[label=below:~]  (x1) at (2,2)  {$\bullet$};
\node[label=below:~ ]  (x2) at (2,0)  {$\bullet$};
\node[label=below:~ ]  (x3) at (0,0)  {$\bullet$};
\node[label=below:~ ]  (x4) at (1,1)  {$\bullet$};
  \node[label=below:~ ]  (x5) at (0,1)  {$\bullet$};

  \draw (x1) .. controls (3,1)  .. (x2);
 \draw (x2) .. controls (1,-0.3)  .. (x3);
 \draw (x3) -- (x4);
 \draw (x4) -- (x5);
  \draw (x5) .. controls (1,3)  .. (x1);

\node[label=below:~]  (z1) at (8,2)  {$\bullet$};
\node[label=below:~ ]  (z2) at (8,0)  {$\bullet$};
\node[label=below:~ ]  (z3) at (6,0)  {$\bullet$};
\node[label=below:~ ]  (z4) at (7,1)  {$\bullet$};

  \draw (z1) .. controls (9,1)  .. (z2);
 \draw (z2) -- (z3);
 \draw[very thick] (z3) -- (z4);

  \draw (z3) .. controls (6,2)  .. (z1);  
  
\draw[->] (3.5,1) -- (5.3,1);  
  
\node [above] at (4.4,1.2) {$u$};

\node [above] at (3,1.2) {$\Gamma_1$};
\node [above] at (10,1.2) {$\Gamma_1=u(\Gamma_1)$};

\node [below] at (1,-0.3) {$\Gamma_2$};
\node [below] at (7.5,-0.3) {$u(\Gamma_2),\;\; h(x)=0$};

\node [above] at (0.2,1.8) {$\Gamma_3$};
\node [above] at (6.2,1.9) {$u(\Gamma_3)$};

\node [above] at (0.8,1.1) {$\Gamma_+$};

\node [below] at (0.8,0.6) {$\Gamma_-$};
\node [below] at (7,0.65) {$u(\Gamma_{\pm})$};

\draw[->] (0.2,0.9) -- (0.3,0.4);  
\node [below] at (0,0.8) {$\sigma$};

\node [below] at (2,1.3) {$\Omega$};
\node [below] at (8,1.3) {$u ({\Omega})$};
  
\end{tikzpicture}
\end{center}
\caption{General scheme of plastic surgery.}
\end{figure}

Let us emphasize the difference between the boundary conditions (C$_{1}$)$-$%
(C$_{3}$) above (see Fig. 1). The first condition (C$_{1}$) means that the surface $%
\Gamma _{1}$ is a part of the elastic body (breast) that remains fixed. The condition (C$_{2}$),
instead, can be interpreted as the knitting of a part of the incised  breast 
(surface $\Gamma _{2}$) to the fixed surface $\mathcal{S}$ of the woman's
chest, while (C$_{3}$) means the knitting of the cut breast surface $\Gamma
_{4}:=\Gamma ^{+}\cup \Gamma ^{-},$ of $\Gamma ^{+}$ into $\Gamma ^{-}$.
Finally, the piece $\Gamma _{3}$ of the boundary $\partial \Omega $ remains
free and admits an arbitrary configuration depending on the knitting process.

\bigskip

\section{\protect\bigskip Existence of minimizers}

Before proving the main existence theorem let us justify that for each $%
\sigma \in \Sigma _{L}\left( \Gamma ^{+};\Gamma ^{-}\right) $ the boundary
condition (C$_{3}$) makes sense.

\begin{lemma}
\label{measurability_comp}Let $\Omega \subset \mathbb{R}^{3}$ be an open
bounded connected domain with locally Lipschitz boundary and $\Gamma ^{\pm
}\subset \partial \Omega $ be $\mathcal{H}^{2}$-measurable sets with $%
\mathcal{H}^{2}\left( \Gamma ^{\pm }\right) >0$ and $\mathcal{H}^{2}\left(
\Gamma ^{+}\cap \Gamma ^{-}\right) =0$. Then for each $u\in \mathbf{W}%
^{1,2}\left( \Omega ;\mathbb{R}^{3}\right) $ and $\sigma \in \Sigma
_{L}\left( \Gamma ^{+};\Gamma ^{-}\right) $, $L\geq 1$, the composed
function $\limfunc{Tr}u\circ \sigma :\Gamma ^{+}\rightarrow \mathbb{R}^{3}$
is measurable w.r.t. the measure $\mathcal{H}^{2}$ on $\Gamma ^{+}$.
\end{lemma}

\begin{proof}
Given $u\in \mathbf{W}^{1,2}\left( \Omega ;\mathbb{R}^{3}\right) $ by the
density argument there exists a sequence of continuous functions $v_{n}:%
\overline{\Omega }\rightarrow \mathbb{R}^{3}$ converging to $u$ in the space 
$\mathbf{W}^{1,2}\left( \Omega ;\mathbb{R}^{3}\right) $. Consequently (see the properties $1$ and $2$ of traces), $%
v_{n}=\limfunc{Tr}v_{n}\rightarrow \limfunc{Tr}u$, as $\ n\rightarrow \infty $,
in $\mathbf{L}^{2}\left( \partial \Omega ;\mathbb{R}^{3}\right) $. Then, up
to a subsequence, we have%
\begin{equation}
v_{n}\left( y\right) \rightarrow \limfunc{Tr}u\left( y\right) \;\;
\forall y\in \Gamma ^{-}\setminus E_{0}^{-}  \label{appr_cont_a.e.}
\end{equation}%
where $E_{0}^{-}\subset \Gamma ^{-}$ is some set with null Hausdorff
measure. So, it remains to prove that%
\begin{equation}
\mathcal{H}^{2}\left( \sigma ^{-1}\left( E_{0}^{-}\right) \right) =0
\label{measure_0}
\end{equation}%
$\,$because in such case we deduce from (\ref{appr_cont_a.e.}) that the
sequence of (continuous) functions $\left\{ v_{n}\left( \sigma \left(
x\right) \right) \right\} $ converges to $\limfunc{Tr}u\left( \sigma \left(
x\right) \right) $ for all $x\in \Gamma ^{+}$ up to the negligible set of
points $E_{0}^{+}:=\sigma ^{-1}\left( E_{0}^{-}\right) $.

On the other hand, (\ref{measure_0}) follows easily from the left inequality
in (\ref{Lipschitz}) and from the definition of \emph{Hausdorff measure}
(see \cite[p. 171]{F}):%
\begin{equation*}
\mathcal{H}^{2}\left( E\right) :=\frac{\pi }{4}\lim_{\varepsilon \rightarrow
0}\,\inf \,\sum_{i=1}^{\infty }\left( \limfunc{diam}A_{i}\right) ^{2}
\end{equation*}%
where the infimum is taken over all coverings $\left\{ A_{i}\right\}
_{i=1}^{\infty }$ of $E$ with the diameters $\limfunc{diam}A_{i}\leq
\varepsilon $. In fact, given $\eta >0$ we find $\varepsilon >0$ and a
family $\left\{ A_{i}\right\} _{i=1}^{\infty }$ with $\tbigcup_{i=1}^{\infty
}A_{i}\supset E_{0}^{-}$ , $\limfunc{diam}A_{i}\leq \varepsilon /L$ and%
\begin{equation*}
\sum_{i=1}^{\infty }\left( \limfunc{diam}A_{i}\right) ^{2}\leq \frac{\eta }{%
L^{2}}.
\end{equation*}%
Since due to (\ref{Lipschitz}) obviously $\limfunc{diam}\,\left( \sigma
^{-1}\left( A_{i}\right) \right) \leq L\limfunc{diam}\,A_{i}\leq \varepsilon 
$, $i=1,2,\dots $; 
\\the family $\left\{ \sigma ^{-1}\left( A_{i}\right) \right\}
_{i=1}^{\infty }$ is a covering of $E_{0}^{+}$ and%
\begin{equation*}
\sum_{i=1}^{\infty }\left( \limfunc{diam}\,\left( \sigma ^{-1}\left(
A_{i}\right) \right) \right) ^{2}\leq L^{2}\sum_{i=1}^{\infty }\left( 
\limfunc{diam}A_{i}\right) ^{2}\leq \eta ,
\end{equation*}%
we conclude that $\mathcal{H}^{2}\left( E_{0}^{+}\right) =0$, and the $%
\mathcal{H}^{2}$-measurability of $x\mapsto \limfunc{Tr}u\left( \sigma
\left( x\right) \right) $ on $\Gamma ^{+}$ follows.
\end{proof}

Let us give now an \emph{a priori} estimate for the "weighted" integral%
\begin{equation*}
\dint\limits_{\Gamma ^{+}}\left\vert \limfunc{Tr}u\left( \sigma \left(
x\right) \right) \right\vert ^{2}d\mathcal{H}^{2}\left( x\right) ,
\end{equation*}%
implying, in particular, that the composed function $\limfunc{Tr}u\circ
\sigma $ belongs to the class $\mathbf{L}^{2}\left( \partial \Omega ;\mathbb{%
R}^{3}\right) $.

\begin{lemma}
\label{estimate_w}Let $\Omega \subset \mathbb{R}^{3}$ and $\Gamma ^{\pm
}\subset \partial \Omega $ be such as in Lemma \ref{measurability_comp}. Then
given $L\geq 1$ there exists a constant $\mathfrak{L}_{L}>0$ such that the inequality%
\begin{equation}
\int_{\Gamma ^{+}}\left\vert \limfunc{Tr}u\left( \sigma \left( x\right)
\right) \right\vert ^{2}d\mathcal{H}^{2}\left( x\right) \leq \mathfrak{L}_{L}%
\int_{\Gamma ^{-}}\left\vert \limfunc{Tr}u\left( y\right) \right\vert ^{2}d%
\mathcal{H}^{2}\left( y\right)   \label{estimate_traces}
\end{equation}%
holds whenever $u\in \mathbf{W}^{1,2}\left( \Omega ;\mathbb{R}^{3}\right) $
and $\sigma \in \Sigma _{L}\left( \Gamma ^{+};\Gamma ^{-}\right) $.
\end{lemma}

\begin{proof}
To prove the estimate (\ref{estimate_traces}) we employ the local
lipschitzianity of the surfaces $\Gamma ^{\pm }$. Namely, given $y\in \Gamma
^{-}$ we choose an (open) ball $B\left( y,\varepsilon _{y}\right) $, $%
\varepsilon _{y}>0$, such that $\Gamma ^{-}\cap B\left( y,\varepsilon
_{y}\right) $ can be represented as the graph of a Lipschitz continuous
function with respect to some (local) coordinates. Without loss of
generality we can suppose that this function (say $f_{y}^{-}$) is defined on
an open set $D_{y}^{-}$ from the cartesian product of the first two
components $x^{\prime }:=\left( x_{1},x_{2}\right) $ and admits as values
the component $x_{3}$, i.e., 
\begin{equation}
\label{Gamma_-cap}
\Gamma ^{-}\cap B\left( y,\varepsilon _{y}\right) =\left\{ \left( z^{\prime
},f_{y}^{-}\left( z^{\prime }\right) \right) :z^{\prime }\in
D_{y}^{-}\right\}.
\end{equation}%
By the compactness of $\Gamma^{-}$ one can find a finite number of points $%
y^{1},\dots ,y^{q}\in \Gamma^{-}$ with 
\begin{equation}
\Gamma ^{-}=\Gamma ^{-}\cap \tbigcup\limits_{j=1}^{q}B\left( y^{j},\frac{%
\varepsilon_{_{y^{j}}}}{2}\right) .  \label{repres.Gamma_-}
\end{equation}%
Set $\varepsilon _{j}:=\varepsilon _{y^{j}}$, $D_{j}^{-}:=D_{y^{j}}^{-}$ and 
$f_{j}^{-}\left( z^{\prime }\right) :=f_{y^{j}}^{-}\left( z^{\prime }\right) 
$, $z^{\prime }\in D_{j}^{-}$, $j=1,\dots ,q$.

Similarly, for any $x\in \Gamma ^{+}$ there exist $\delta _{x}>0$, an open
domain $D_{x}^{+}\subset \mathbb{R}^{2}$ and a Lipschitz function $%
f_{x}^{+}:D_{x}^{+}\rightarrow \mathbb{R}$ such that 
\begin{equation*}
\Gamma ^{+}\cap B\left( x,\delta _{x}\right) =\left\{ \left( z^{\prime
},f_{x}^{+}\left( z^{\prime }\right) \right) :z^{\prime }\in
D_{x}^{+}\right\} .
\end{equation*}%
We do not loose generality assuming that the value of $f_{x}^{+}$ is the
last component of the vector $z\in \Gamma ^{+}$ (as in (\ref{Gamma_-cap}%
)). Again due to the compactness argument there exists a finite number of
points $x^{1},\dots ,x^{r}\in \Gamma ^{+}$ such that%
\begin{equation*}
\Gamma ^{+}=\Gamma ^{+}\cap \tbigcup\limits_{i=1}^{r}B\left( x^{i},\delta
_{i}\right) 
\end{equation*}%
where%
\begin{equation}
\delta _{i}:=\underset{1\leq j\leq q}{\min }\left( \delta _{x^{i}},\tfrac{%
\varepsilon _{j}}{2L}\right) ,\;\; i=1,\dots ,r.
\label{def_delta}
\end{equation}%
Set also $D_{i}^{+}:=D_{x^{i}}^{+}$ and $f_{i}^{+}\left( z^{\prime }\right)
:=f_{x^{i}}^{+}\left( z^{\prime }\right) $, $z^{\prime }\in D_{i}^{+}$, $%
i=1,\dots ,r$, and denote by $L_{\Gamma }$ the maximal Lipschitz constant of
the functions $f_{j}^{-}:D_{j}^{-}\rightarrow \mathbb{R}$, $j=1,\dots ,q$,
and $f_{i}^{+}:D_{i}^{+}\rightarrow \mathbb{R}$, $i=1,\dots ,r$.

We claim that given $i=1,\dots ,r$ and $\sigma \in \Sigma _{L}$, the image $%
\sigma \left( \Gamma ^{+}\cap B\left( x^{i},\delta _{i}\right) \right) $ is
contained in $\Gamma ^{-}\cap B\left( y^{j},\varepsilon _{j}\right) $ for
some $j=1,\dots ,q$. Indeed, let us choose $j$ such that $\sigma \left(
x^{i}\right) \in B\left( y^{j},\varepsilon _{j}/2\right) $ (see (\ref%
{repres.Gamma_-})). Taking an arbitrary $z\in \Gamma ^{+}\cap B\left(
x^{i},\delta _{i}\right) $ we, in particular, have $\left\vert
z-x^{i}\right\vert <\frac{\varepsilon _{j}}{2L}$ (see (\ref{def_delta})) and
by (\ref{Lipschitz}) $\left\vert \sigma \left( z\right) -\sigma \left(
x^{i}\right) \right\vert <\varepsilon _{j}/2$. However, assuming that $%
\sigma \left( z\right) \notin \Gamma ^{-}\cap B\left( y^{j},\varepsilon
_{j}\right) $ we have $\left\vert \sigma \left( z\right) -y^{j}\right\vert
\geq \varepsilon _{j}$ and hence%
\begin{equation*}
\left\vert \sigma \left( z\right) -\sigma \left( x^{i}\right) \right\vert
\geq \left\vert \sigma \left( z\right) -y^{j}\right\vert -\left\vert \sigma
\left( x^{i}\right) -y^{j}\right\vert \geq \varepsilon _{j}-\frac{%
\varepsilon _{_{j}}}{2}=\frac{\varepsilon _{_{j}}}{2},
\end{equation*}%
which is a contradiction. In what follows we associate to each $\sigma \in \Sigma _{L}$ and to each $i=1,\dots ,r$ an index $%
j=j\left( \sigma ,i\right) \in \left\{ 1,\dots ,q\right\} $ such that%
\begin{equation*}
\sigma \left( \Gamma ^{+}\cap B\left( x^{i},\delta _{i}\right) \right)
\subset \Gamma ^{-}\cap B\left( y^{j},\varepsilon _{j}\right) .
\end{equation*}

The latter inclusion allows us to define correctly the (injective) mapping $%
\psi _{\sigma }^{i}:D_{i}^{+}\rightarrow D_{j\left( \sigma ,i\right) }^{-}$
such that%
\begin{equation}
\sigma \left( x\right) =\sigma \left( x^{\prime },f_{i}^{+}\left( x^{\prime
}\right) \right) =\left( \psi _{\sigma }^{i}\left( x^{\prime }\right)
,f_{j}^{-}\left( \psi _{\sigma }^{i}\left( x^{\prime }\right) \right)
\right) ,x^{\prime }\in D_{i}^{+}.  \label{def_psi}
\end{equation}%
>From (\ref{Lipschitz}) it follows immediately that $\psi _{\sigma }^{i}$ is
Lipschitz:%
\begin{eqnarray}
\left\vert \psi _{\sigma }^{i}\left( x^{\prime }\right) -\psi _{\sigma
}^{i}\left( z^{\prime }\right) \right\vert  &\leq &\left\vert \sigma \left(
x^{\prime },f_{i}^{+}\left( x^{\prime }\right) \right) -\sigma \left(
z^{\prime },f_{i}^{+}\left( z^{\prime }\right) \right) \right\vert   \notag
\\
&\leq &L\left( \left\vert x^{\prime }-z^{\prime }\right\vert ^{2}+\left(
f_{i}^{+}\left( x^{\prime }\right) -f_{i}^{+}\left( z^{\prime }\right)
\right) ^{2}\right) ^{1/2}  \label{Lip_psi} \\
&\leq &L\sqrt{1+L_{\Gamma }^{2}}\left\vert x^{\prime }-z^{\prime
}\right\vert ,\;\; x^{\prime },z^{\prime }\in D_{i}^{+}. 
\notag
\end{eqnarray}%
So, by \emph{Rademacher's Theorem,} $\psi _{\sigma }^{i}$ is $\mathcal{L}^{2}
-$a.e. differentiable on $D_{i}^{+}$ with $\mathcal{L}^{2}$-measurable
gradient, and the inequality%
\begin{equation*}
\left\vert \nabla \psi _{\sigma }^{i}\left( x^{\prime }\right) \right\vert
\leq M:=L\sqrt{1+L_{\Gamma }^{2}}  \label{bound_grad_psi}
\end{equation*}%
holds for a.e. $x^{\prime }\in D_{i}^{+}$. Notice that the inverse mapping $%
\left( \psi _{\sigma }^{i}\right) ^{-1}$ is well defined on $G_{\sigma
}^{i}:=\psi _{\sigma }^{i}\left( D_{i}^{+}\right) \subset D_{j\left( \sigma
,i\right) }^{-}$ by the formula similar to (\ref{def_psi}), namely,%
\begin{equation*}
\sigma ^{-1}\left( y^{\prime },f_{j}^{-}\left( y^{\prime }\right) \right)
=\left( \left( \psi _{\sigma }^{i}\right) ^{-1}\left( y^{\prime }\right)
,f_{i}^{+}\left( \left( \psi _{\sigma }^{i}\right) ^{-1}\left( y^{\prime
}\right) \right) \right) ,y^{\prime }\in G_{\sigma }^{i}.
\end{equation*}%
In the same way as (\ref{Lip_psi}) we deduce, from (\ref{Lipschitz}), that $%
\left( \psi _{\sigma }^{i}\right) ^{-1}$ is Lipschitz on the (open) set $%
G_{\sigma }^{i}$ and, the estimate%
\begin{equation}
\left\vert \nabla \left( \psi _{\sigma }^{i}\right) ^{-1}\left( y^{\prime
}\right) \right\vert \leq M  \label{bound_grad_psi_inv}
\end{equation}%
holds for $\mathcal{L}^{2}$-a.e. $y^{\prime }\in G_{\sigma }^{i}$.

Integrating the function $\left\vert \limfunc{Tr}u\left( \sigma \left(
x\right) \right) \right\vert ^{2}$ on the surface piece $\Gamma ^{+}\cap
B\left( x^{i},\delta _{i}\right) $ we pass first to the double integral%
\begin{eqnarray}
&&\underset{\Gamma ^{+}\cap B\left( x^{i},\delta _{i}\right) }{\int }%
\left\vert \limfunc{Tr}u\left( \sigma \left( x\right) \right) \right\vert
^{2}d\mathcal{H}^{2}\left( x\right)   \notag \\
&=&\diint\limits_{D_{i}^{+}}\left\vert \limfunc{Tr}u\left( \sigma \left(
x^{\prime },f_{i}^{+}\left( x^{\prime }\right) \right) \right) \right\vert
^{2}\sqrt{1+\left\vert \nabla f_{i}^{+}\left( x^{\prime }\right) \right\vert
^{2}}dx^{\prime }  \label{int_change_var} \\
&\leq &\sqrt{1+L_{\Gamma }^{2}}\diint\limits_{D_{i}^{+}}\left\vert \limfunc{%
Tr}u\left( \sigma \left( x^{\prime },f_{i}^{+}\left( x^{\prime }\right)
\right) \right) \right\vert ^{2}\,dx^{\prime }.  \notag
\end{eqnarray}%
Due to the representation (\ref{def_psi}) we can make the change of
variables $y^{\prime }=\psi _{\sigma }^{i}\left( x^{\prime }\right) $, $%
x^{\prime }\in D_{i}^{+}$, in the integral (\ref{int_change_var}), and
returning then to the surface integral on a piece of $\Gamma ^{-}$, we have%
\begin{eqnarray}
&&\diint\limits_{D_{i}^{+}}\left\vert \limfunc{Tr}u\left( \sigma \left(
x^{\prime },f_{i}^{+}\left( x^{\prime }\right) \right) \right) \right\vert
^{2}\,dx^{\prime }  \notag \\
&=&\diint\limits_{G_{\sigma }^{i}}\left\vert \limfunc{Tr}u\left( y^{\prime
},f_{j}^{-}\left( y^{\prime }\right) \right) \right\vert ^{2}\left\vert \det
\nabla \left( \psi _{\sigma }^{i}\right) ^{-1}\left( y^{\prime }\right)
\right\vert \,dy^{\prime }  \notag \\
&\leq &6M^{3}\diint\limits_{G_{\sigma }^{i}}\left\vert \limfunc{Tr}u\left(
y^{\prime },f_{j}^{-}\left( y^{\prime }\right) \right) \right\vert ^{2}\sqrt{%
1+\left\vert \nabla f_{j}^{-}\left( y^{\prime }\right) \right\vert ^{2}}%
dy^{\prime }  \label{change_var1} \\
&\leq &6M^{3}\underset{\Gamma ^{-}}{\int }\left\vert \limfunc{Tr}u\left(
y\right) \right\vert ^{2}d\mathcal{H}^{2}\left( y\right) .  \notag
\end{eqnarray}%
Here we used the estimate (\ref{bound_grad_psi_inv}) and the obvious
inequality $\left\vert \det A\right\vert \leq 6\left\Vert A\right\Vert ^{3}$
($A$ is an arbitrary $3\times 3$-matrix). Since the sets $\Gamma ^{+}\cap
B\left( x^{i},\delta _{i}\right) $, $i=1,\dots ,r$, cover the surface $%
\Gamma ^{+}$, taking into account (\ref{int_change_var}) and (\ref%
{change_var1}) we conclude that%
\begin{eqnarray*}
\int_{\Gamma ^{+}}\left\vert \limfunc{Tr}u\left( \sigma \left( x\right)
\right) \right\vert ^{2}d\mathcal{H}^{2}\left( x\right)  &\leq &\underset{i=1%
}{\overset{r}{\dsum }}\underset{\Gamma ^{+}\cap B\left( x^{i},\delta
_{i}\right) }{\int }\left\vert \limfunc{Tr}u\left( \sigma \left( x\right)
\right) \right\vert ^{2}d\mathcal{H}^{2}\left( x\right)  \\
&\leq &\mathfrak{L}\underset{\Gamma ^{-}}{\int }\left\vert \limfunc{Tr}%
u\left( y\right) \right\vert ^{2}d\mathcal{H}^{2}\left( y\right) 
\end{eqnarray*}%
where $\mathfrak{L}:=6rM^{3}\sqrt{1+L_{\Gamma }^{2}}>0$ depends just on the
Lipschitz constant $L\geq 1$ and on the properties of the domain $\Omega $
(namely, of its boundary).\bigskip 
\end{proof}

\medskip

Proving the existence theorem we pay the main attention to the validity of
the boundary condition (C$_{3}$) where Lemma \ref{estimate_w} is crucial.

\begin{theorem}
\label{Th_exist} Let $W:\mathbb{R}^{3\times 3}\rightarrow \mathbb{R\cup }%
\left\{ +\infty \right\} $ be a polyconvex function satisfying the growth
assumption (\ref{growth_cond}). Then problem (\ref{varproblem}) admits a
minimizer whenever there exists at least one pair $\omega :=\left( u,\sigma
\right) \in \mathcal{W}_{L}$ with%
\begin{equation*}
I\left( \omega \right) :=\dint_{\Omega }W\left( \nabla u\left( x\right)
\right) \,\,dx<+\infty .
\end{equation*}
\end{theorem}

\begin{proof}
Let us consider a minimizing sequence $\left\{ \left( u_{n},\sigma
_{n}\right) \right\} \subset \mathcal{W}_{L}$ of the functional (\ref%
{Functional}), e.g., such as%
\begin{equation}
\dint_{\Omega }W\left( \nabla u_{n}\left( x\right) \right) \,dx\leq \inf
\left\{ \dint_{\Omega }W\left( \nabla u\left( x\right) \right) \,dx:\left(
u,\sigma \right) \in \mathcal{W}_{L}\right\} +\frac{1}{n}<+\infty .
\label{min_seq}
\end{equation}%
Taking into account the estimate (\ref{growth_cond}) we deduce from (\ref%
{min_seq}) that the sequences $\left\{ \nabla u_{n}\right\} $, $\left\{ 
\mathrm{Adj\,}\nabla u_{n}\right\} $ and $\left\{ \det \,\nabla
u_{n}\right\} $ are bounded in $\mathbf{L}^{2}\left( \Omega ;\mathbb{R}%
^{3\times 3}\right) $ and in $\mathbf{L}^{2}\left( \Omega ;\mathbb{R}\right) 
$, respectively. Applying Proposition \ref{poincare} and the boundary
condition (C$_{1}$) we find a constant $C>0$ such that the inequality%
\begin{equation*}
\underset{\Omega }{\dint }\left\vert u_{n}\left( x\right) \right\vert
^{2}dx\leq C\left[ \underset{\Omega }{\dint }\left\vert \nabla u_{n}\left(
x\right) \right\vert ^{2}\,dx+\left\vert \dint\limits_{\Gamma _{1}}x\,d%
\mathcal{H}^{2}\left( x\right) \right\vert \right] 
\end{equation*}%
holds for each  $n\geq 1$. So, the sequence $\left\{ u_{n}\right\}$ is bounded in $\mathbf{W}^{1,2}\left( \Omega ;\mathbb{R}%
^{3}\right) $ and by the \emph{Banach-Alaoglu theorem}, up to a subsequence,
converges weakly to some function $\bar{u}\in \mathbf{W}^{1,2}\left( \Omega ;%
\mathbb{R}^{3}\right) $. Without loss of generality, we can also assume that 
$\left\{ \mathrm{Adj\,}\nabla u_{n}\right\} $ and $\left\{ \mathrm{\det \,}%
\nabla u_{n}\right\} $ converge weakly to some functions $\xi \in \mathbf{L}%
^{2}\left( \Omega ;\mathbb{R}^{3\times 3}\right) $ and $\eta \in \mathbf{L}%
^{2}\left( \Omega ;\mathbb{R}\right) $, respectively. Now, by Theorem 8.20 
\cite[pp. 395-396]{Dac} due to the uniqueness of the limit we deduce that $%
\xi \left( x\right) =\mathrm{Adj\,}\nabla u\left( x\right) $ and $\eta
\left( x\right) =\mathrm{\det \,}\nabla u\left( x\right) $ for almost all $%
x\in \Omega $. Thus we have the weak convergence\ of the sequence $\left\{ 
\mathbb{T}\left( \nabla u_{n}\right) \right\} $ to the vector-function $%
\mathbb{T}\left( \nabla u\right) $ in the space $\mathbf{L}^{2}\left( \Omega
;\mathbb{R}^{\tau \left( 3,3\right) }\right) $.

On the other hand, recalling that $\left\{ \sigma _{n}\right\} \subset
\Sigma _{L}\left( \Gamma ^{+};\Gamma ^{-}\right) $ (see (\ref{Lipschitz}))
by \emph{Ascoli's theorem} up to a subsequence, not relabeled, $\left\{
\sigma _{n}\right\} $ converges uniformly to $\overline{\sigma }\in \Sigma
_{L}\left( \Gamma ^{+};\Gamma ^{-}\right) $.

Since the integrand $W$ is polyconvex, it can be represented as $W\left( \xi
\right) =g\left( \mathbb{T}\left( \xi \right) \right) $, $\xi \in \mathbb{R}%
^{3\times 3}$, with some convex function $g:\mathbb{R}^{\tau \left(
3,3\right) }\rightarrow \mathbb{R}$, and, therefore,%
\begin{eqnarray*}
\underset{\Omega }{\dint }W\left( \nabla \overline{u}\left( x\right) \right)
\,\,dx &=&\underset{\Omega }{\dint }g\left( \mathbb{T}\left( \nabla 
\overline{u}\left( x\right) \right) \right) \,\,dx\leq \underset{%
n\rightarrow \infty }{\lim \inf }\,\underset{\Omega }{\dint }g\left( \mathbb{%
T}\left( \nabla u_{n}\left( x\right) \right) \right) \,\,dx \\
&\leq &\inf \left\{ \underset{\Omega }{\dint }W\left( \nabla u\left(
x\right) \right) \,\,dx:\left( u,\sigma \right) \in \mathcal{W}_{L}\right\} 
.
\end{eqnarray*}%
Thus, it remains just to prove that $\bar{\omega}:=\left( \overline{u},%
\overline{\sigma }\right) \in \mathcal{W}_{L}$ (i.e., that the Sobolev
function $\overline{u}$ satisfies the boundary conditions (C$_{1}$)$-$(C$_{3}
$) above with the transformation $\overline{\sigma }$). The validity of (C$%
_{1}$) and (C$_{2}$) follows immediately from Proposition \ref{traceconv}.
In fact, the weak convergence of $\left\{ u_{n}\right\} $ in $\mathbf{W}%
^{1,2}\left( \Omega ;\mathbb{R}^{3}\right) $ implies the strong convergence
of traces $\left\{ \limfunc{Tr}u_{n}\right\} $ in $\mathbf{L}^{2}\left(
\partial \Omega ;\mathbb{R}^{3}\right) $. So, up to a subsequence, $\limfunc{%
Tr}u_{n}\left( x\right) \rightarrow \limfunc{Tr}\bar{u}\left( x\right) $ for 
$\mathcal{H}^{2}$-a.e. $x\in \partial \Omega $. In particular, $\limfunc{Tr}%
\bar{u}\left( x\right) =x$ and $h\left( \limfunc{Tr}\bar{u}\left( x\right)
\right) =0$ almost everywhere on $\Gamma _{1}$ and on $\Gamma _{2}$,
respectively (w.r.t. the Hausdorff measure).

In order to verify the condition (C$_{3}$) we observe first that%
\begin{equation}
\limfunc{Tr}{u}_{n}(x) =\limfunc{Tr}{u}_{n}({
\sigma}_{n}(x)) ,n=1,2,\dots ,
\label{cond_C3_n}
\end{equation}%
for $\mathcal{H}^{2}$-a.e. $x\in \Gamma ^{+}$ and consider the surface
integral%
\begin{equation*}
\mathcal{J}:=\int_{\Gamma ^{+}}\left\vert \limfunc{Tr}\bar{u}\left( x\right)
-\limfunc{Tr}\bar{u}\left( \bar{\sigma}\left( x\right) \right) \right\vert
^{2}d\mathcal{H}^{2}\left( x\right) .
\end{equation*}%
By the Minkowski's inequality we have%
\begin{equation}
\mathcal{J}^{1/2}\leq \left( \mathcal{J}_{1}^{n}\right) ^{1/2}+\left( 
\mathcal{J}_{2}^{n}\right) ^{1/2}+\left( \mathcal{J}_{3}^{n}\right) ^{1/2}
\label{estimate_int}
\end{equation}%
where%
\begin{eqnarray*}
&&\mathcal{J}_{1}^{n}:=\int_{\Gamma ^{+}}\left\vert \limfunc{Tr}\bar{u}%
\left( x\right) -\limfunc{Tr}u_{n}\left( \sigma _{n}\left( x\right) \right)
\right\vert ^{2}d\mathcal{H}^{2}\left( x\right)  {;} \\
&&\mathcal{J}_{2}^{n}:=\int_{\Gamma ^{+}}\left\vert \limfunc{Tr}u_{n}\left(
\sigma _{n}\left( x\right) \right) -\limfunc{Tr}\bar{u}\left( \sigma
_{n}\left( x\right) \right) \right\vert ^{2}d\mathcal{H}^{2}\left( x\right) 
 {;} \\
&&\mathcal{J}_{3}^{n}:=\int_{\Gamma ^{+}}\left\vert \limfunc{Tr}\bar{u}%
\left( \sigma _{n}\left( x\right) \right) -\limfunc{Tr}\bar{u}\left( \bar{%
\sigma}\left( x\right) \right) \right\vert ^{2}d\mathcal{H}^{2}\left(
x\right) .
\end{eqnarray*}

Taking into account the equalities (\ref{cond_C3_n}) by using Proposition %
\ref{traceconv} we immediately obtain that $\mathcal{J}_{1}^{n}\rightarrow 0$
as $n\rightarrow \infty $.

Due to the linearity of the trace operator, applying Lemma \ref{estimate_w}
and again Proposition \ref{traceconv} we arrive at%
\begin{equation*}
\mathcal{J}_{2}^{n}\leq \mathfrak{L}_L\int_{\Gamma ^{-}}\left\vert \limfunc{Tr}%
\left( u_{n}-\bar{u}\right) \left( x\right) \right\vert ^{2}d\mathcal{H}%
^{2}\left( x\right) \rightarrow 0 { \ \ as \ \ }n\rightarrow \infty 
.
\end{equation*}

Let us approximate now $\bar{u}\in \mathbf{W}^{1,2}\left( \Omega ;\mathbb{R}%
^{3}\right) $ by a sequence of continuous functions $v_{k}:\overline{\Omega }%
\rightarrow \mathbb{R}^{3}$, $k=1,2,\dots $ (with respect to the norm of $%
\mathbf{W}^{1,2}\left( \Omega ;\mathbb{R}^{3}\right) $). Then (see
Proposition \ref{traceconv}) $v_{k}=\limfunc{Tr}v_{k}\rightarrow \limfunc{Tr}%
\bar{u}$ as $k\rightarrow \infty$ in $\mathbf{L}^{2}\left( \partial \Omega
;\mathbb{R}^{3}\right) $. In particular, given $\varepsilon >0$ there exists
an index $k^{\ast }\geq 1$ such that%
\begin{equation*}
\int_{\Gamma ^{+}}\left\vert \limfunc{Tr}\bar{u}\left( y\right) -v_{k^{\ast
}}\left( y\right) \right\vert ^{2}d\mathcal{H}^{2}\left( y\right) \leq
\varepsilon .
\end{equation*}%
By using Lemma \ref{estimate_w} similarly as was done to estimate the
integral $\mathcal{J}_{2}^{n}$ we have%
\begin{eqnarray}
&&\int_{\Gamma ^{+}}\left\vert \limfunc{Tr}\bar{u}\left( \sigma _{n}\left(
x\right) \right) -v_{k^{\ast }}\left( \sigma _{n}\left( x\right) \right)
\right\vert ^{2}d\mathcal{H}^{2}\left( x\right)   \notag \\
&\leq &\mathfrak{L}_{L}\int_{\Gamma ^{-}}\left\vert \limfunc{Tr}\bar{u}\left(
y\right) -v_{k^{\ast }}\left( y\right) \right\vert ^{2}d\mathcal{H}%
^{2}\left( y\right) \leq \mathfrak{L}_{L}\varepsilon ,\;\; n=1,2,\dots 
,  \label{Int_3_1}
\end{eqnarray}%
and similarly
\begin{equation}
\int_{\Gamma ^{+}}\left\vert \limfunc{Tr}\bar{u}\left( \bar{\sigma}\left(
x\right) \right) -v_{k^{\ast }}\left( \bar{\sigma}\left( x\right) \right)
\right\vert ^{2}d\mathcal{H}^{2}\left( x\right) \leq \mathfrak{L}_{L}\varepsilon 
.  \label{Int_3_2}
\end{equation}%
On the other hand, by the uniform continuity of $v_{k^{\ast }}$ and the uniform
convergence $\sigma _{n}\rightarrow \overline{\sigma }$ as $n\rightarrow
\infty $, we find a number $n^{\ast }\geq 1$ such that%
\begin{equation*}
\left\vert v_{k^{\ast }}\left( \sigma _{n}\left( x\right) \right)
-v_{k^{\ast }}\left( \bar{\sigma}\left( x\right) \right) \right\vert \leq
\varepsilon 
\end{equation*}%
for all $n\geq n^{\ast }$ and all $x\in \Gamma ^{+}$, and, consequently,%
\begin{equation}
\int_{\Gamma ^{+}}\left\vert v_{k^{\ast }}\left( \sigma _{n}\left( x\right)
\right) -v_{k^{\ast }}\left( \bar{\sigma}\left( x\right) \right) \right\vert
^{2}d\mathcal{H}^{2}\left( x\right) \leq \mathcal{H}^{2}\left( \Gamma
^{+}\right) \varepsilon ^{2},n\geq n^{\ast }.  \label{Int_3_3}
\end{equation}%
Joining together the inequalities (\ref{Int_3_1}), (\ref{Int_3_2}) and (\ref%
{Int_3_3}) we obtain that%
\begin{equation*}
\left( \mathcal{J}_{3}^{n}\right) ^{1/2}\leq \left( \mathfrak{L}_L\varepsilon
\right) ^{1/2}+\left( \mathfrak{L}_L\varepsilon \right) ^{1/2}+\left( \mathcal{%
H}^{2}\left( \Gamma ^{+}\right) \varepsilon ^{2}\right) ^{1/2},n\geq
n^{\ast }.
\end{equation*}%
Since $\varepsilon >0$ is arbitrary and the constant $\mathfrak{L}_L$ does not
depend on $n=1,2,\dots $, we conclude that all the three integrals in the
right-hand side of (\ref{estimate_int}) tend to zero as $n\rightarrow \infty 
$. Thus $\mathcal{J}=0$, or, in other words, $\limfunc{Tr}\bar{u}\left(
x\right) -\limfunc{Tr}\bar{u}\left( \bar{\sigma}\left( x\right) \right) =0$
for $\mathcal{H}^{2}$-a.e. $x\in \Gamma ^{+}$, and the theorem is proved.
\end{proof}

\section{Necessary conditions of optimality}

In this section, under some additional hypotheses, we deduce necessary conditions of optimality for problem (\ref{varproblem}).

To simplify, assume that the function $W$ is twice continuously differentiable  and $h$ is continuously differentiable.
Moreover, suppose that the surfaces $\Gamma_1, \Gamma_2, \Gamma^+, \Gamma^-, \Gamma_4$ are sufficiently smooth. 

Given $\Gamma \subset \partial\Omega,$ with $\mathcal{H}^{2}(\Gamma)>0,$ in what follows we denote by $\mathbf{C}^{1}(\Gamma^+;\mathbb{R}^{3})$ the family of restrictions to $\Gamma$ of all functions $u:\Omega\rightarrow \mathbb{R}^{3},$ whose gradients are continuous up to the boundary. Let us supply $\mathbf{C}^{1}(\Omega;\mathbb{R}^{3})$ with the natural sup-norm. 

We consider the problem (\ref{varproblem}) defined in the space $\mathbf{C}^{1}(\Omega;\mathbb{R}^{3})\times \mathbf{C}^{1}(\Gamma^+;\mathbb{R}^{3})$.

\begin{theorem}
Let $(\bar{u},\bar{\sigma})\in \mathbf{C}^{1}({\Omega};\mathbb{R}^{3})\times \mathbf{C}^{1}({\Gamma}^+;\mathbb{R}^{3})$ be a minimizer of problem (\ref{varproblem}). Assume that $\nabla h(\bar{u}(x))\neq 0$, $x\in \Gamma_2$ and ${\rm det}\nabla\bar{u}(\bar{\sigma}(x))\neq 0$, $x\in\Omega$. Then the following conditions are satisfied:
\begin{eqnarray}
&& {\rm Div}(\nabla W)(\nabla\bar{u}(x))=0,\;\; x\in\Omega; \label{f1}\\
\smallskip
&& \nabla W(\nabla\bar{u}(x))\nu(x)=0,\;\; x\in \Gamma_3; \label{f2}\\
\smallskip
&& \nabla W(\nabla\bar{u}(x))\nu(x)\times\nabla h(\bar{u}(x))=0,\;\; x\in \Gamma_2; \label{f3}\\
\smallskip
&&\nabla W(\nabla\bar{u}(x))\nu(x)=0,\;\; x\in \Gamma^{\pm}. \label{f4}
\end{eqnarray}
\end{theorem}

\begin{proof} Let us write the constraints in the minimization problem (\ref{varproblem}) as $F(u,\sigma)=0$ where the map  
$
F:\mathbf{C}^{1}({\Omega};\mathbb{R}^{3})\times \mathbf{C}^{1}({\Gamma}^+;\mathbb{R}^{3})\rightarrow \mathbf{C}^{1}({\Gamma}_1;\mathbb{R}^{3})\times \mathbf{C}^{1}({\Gamma}_2;\mathbb{R})\times \mathbf{C}^{1}({\Gamma}^+;\mathbb{R}^{3})$ is given by
$$
F(u,\sigma ):=(u(x)-x,h(u(x)),u(x)-u(\sigma(x))).
$$
Under our assumptions the map $F$ and the functional $I$ are both Fr\'{e}chet differentiable. In particular, for the (Fr\'{e}chet) derivative of $F$ at the point $(\bar{u},\bar{\sigma})$ we have
$$
DF(\bar{u},\bar{\sigma})(\tilde{u},\tilde{\sigma})(x)=
\left(
\begin{array}{c}
\tilde{u}(x)\\
\langle\nabla h(\bar{u}(x)),\tilde{u}(x)\rangle\\
\tilde{u}(x)-\tilde{u}(\bar{\sigma}(x))-\nabla \bar{u}(\bar{\sigma}(x))\tilde{\sigma}(x)
\end{array}
\right).
$$
Here, and in what follows, $\langle\cdot,\cdot\rangle$ denotes the inner product in $\mathbb{R}^{3}.$ Taking into account that $\nabla h(\bar{u}(x))\neq0$ on $\Gamma_2$ and that the jacobian matrix $\nabla\bar{u}(\bar{\sigma}(x))$ is not degenerated, we have that the linear operator $DF(\bar{u},\bar{\sigma})$ is onto the space $\mathbf{C}^{1}({\Gamma}_1;\mathbb{R}^{3})\times \mathbf{C}^{1}(\Gamma_2;\mathbb{R})\times \mathbf{C}^{1}(\Gamma^+;\mathbb{R}^{3}).$
By the Lagrange multipliers rule (see, e.g., \cite{IT}) there exist linear continuous functionals $\lambda_1,\;\lambda_2,\;\lambda^+$ on $\mathbf{C}^{1}(\Gamma_1;\mathbb{R}^{3})$,\; $\mathbf{C}^{1}(\Gamma_2;\mathbb{R})$ and $\mathbf{C}^{1}(\Gamma^+;\mathbb{R}^{3}),$ respectively, such that
\begin{equation*}
\int_{\Omega}\sum_{i,j=1}^{3}\frac{\partial W(\nabla\bar{u}(x))} {\partial\xi_{ij}}\frac{\partial\tilde{u}_i(x)}{\partial x_j}dx
+\lambda_1(\tilde{u})+\lambda_2(\langle\nabla h(\bar{u}),\tilde{u}\rangle)
+\lambda^+(\tilde{u}-\tilde{u}(\bar{\sigma})-\nabla \bar{u}(\bar{\sigma})\tilde{\sigma})=0.
\end{equation*}

\noindent Applying the Divergence theorem we get
\begin{eqnarray}
&&-\int_{\Omega}\left\langle {\rm Div}{(\nabla W(\nabla\bar{u}(x)))},\tilde{u}(x)\right\rangle dx
+\int_{\partial\Omega}\left\langle{\nabla W(\nabla\bar{u}(x))\nu(x),\tilde{u}(x)}\right\rangle d\mathcal{H}^{2}(x) \nonumber \\
&&+\lambda_1(\tilde{u})+\lambda_2(\langle\nabla h(\bar{u}),\tilde{u}\rangle)+\lambda^+(\tilde{u}-\tilde{u}(\bar{\sigma})-\nabla \bar{u}(\bar{\sigma})\tilde{\sigma}))=0. \label{f5} 
\end{eqnarray}
Here $\nu(x)$ is the unit outer normal to the boundary. Varying $\tilde{u}$ in (\ref{f5}) such that $\tilde{u}(x)=0$ on $\partial{\Omega}$, we obtain  (\ref{f1}). Taking then $\tilde{u} \in \mathbf{C}^{1}(\Omega;\mathbb{R}^{3})$ with $\tilde{u}(x)=0$ on $\partial\Omega\setminus \Gamma_3$ we arrive at (\ref{f2}).
Furthermore, choosing appropriate functions $\tilde{u}$ in (\ref{f5}) we obtain
\begin{equation}
\label{f6}
\int_{\Gamma_2}\left\langle\nabla W(\nabla\bar{u}(x))\nu (x),\tilde{u}(x)\right\rangle d\mathcal{H}^{2}(x)
+\lambda_2(\langle\nabla h(\bar{u}),\tilde{u}\rangle)=0 
\end{equation}
whenever $\tilde{u}(x)=0,\; x\in\partial\Omega\setminus \Gamma_2$, and
\begin{equation}
\label{f7}
\int_{\Gamma^+\cup\Gamma^-}\left\langle{\nabla W(\nabla\bar{u}(x)) }\nu (x),\tilde{u}(x)\right\rangle d\mathcal{H}^{2}(x)
+\lambda^+(\tilde{u}-\tilde{u}(\bar{\sigma})-\nabla \bar{u}(\bar{\sigma})\tilde{\sigma})=0\;\; 
\end{equation}
whenever $\tilde{u}(x)=0,\;\; x\in\partial\Omega\setminus (\Gamma^+\cup\Gamma^-)$.

Denote by $\Gamma_2^0$ the part of $\Gamma_2$ where the vectors $a(x):=\nabla W(\nabla\bar{u}(x)) \nu (x)$ and
$b(x):=\nabla h(\bar{u})(x)$ are co-linear. Taking an arbitrary $c \in \mathbf{C}^{1}(\Gamma_2;\mathbb{R})$ such that $c(x)=0$ and $\nabla c(x)=0$ on $\Gamma_{2}^0$, let us define
$$
\hat{u}(x):=
\left\{
\begin{array}{cl}
\frac{a(x)\langle a(x),b(x)\rangle-b(x)|a(x)|^2}{\langle a(x),b(x)\rangle^2-|a(x)|^2|b(x)|^2}c(x),& x\in\Gamma_2\setminus\Gamma_2^0,\\
\smallskip
0, & x\in\Gamma_2^0. 
\end{array}
\right.
$$
Obviously, $\hat{u} \in \mathbf{C}^{1}(\Gamma_2;\mathbb{R}^3)$, $\langle\hat{u}(x), a(x)\rangle=0$ and $\langle\hat{u}(x),b(x)\rangle=c(x)$ for $x\in \Gamma_2$. From (\ref{f6}) we
get $\lambda_2(c)=\lambda_2(\langle b,\hat{u}\rangle)=0$. Hence, varying $\tilde{u} \in \mathbf{C}^{1}(\Gamma_2;\mathbb{R}^3)$ in (\ref{f6}) in a suitable way (in particular, setting $\hat u(x)=0$ on $\Gamma_2^0$) we get $a(x)=0$ in $\Gamma_2\setminus\Gamma_2^0.$ Thus the equality (\ref{f3}) follows.

Finally, taking $\tilde{u}=0$, from (\ref{f7}) we get $\lambda^+=0$, and, as a consequence, 
$$
\int_{\Gamma^-}\left\langle\nabla W (\nabla\bar{u}(x))\nu (x),\tilde{u}(x)\right\rangle d\mathcal{H}^2(x)
+\int_{\Gamma^+}\left\langle\nabla W (\nabla\bar{u}(x))\nu (x),\tilde{u}(x)\right\rangle d\mathcal{H}^2(x)=0,
$$
which implies (\ref{f4}). 
\end{proof}

\bigskip

\bigskip

\noindent  {\bf Acknowledgements}

\bigskip 

The authors are grateful to Hor\'{a}cio Costa and Augusta Cardoso for fruitful discussion 
of medical aspects of the problem and also to Giovanni Leoni, who kindly communicate the proof of Lemma \ref
{Lemmatrace}.

This research was supported by Funda\c{c}\~{a}o para a Ci\^{e}ncia e
Tecnologia (FCT), Portuguese Operational Programme for Competitiveness
Factors (COMPETE), Portuguese National Strategic Reference Framework (QREN)
and European Regional Development Fund (FEDER) through Project VAPS
(EXPL/MAT-NAN/0606/2013).

\end{document}